\documentclass[final]{siamltex}
\usepackage{amsfonts}
\usepackage{amsmath}
\usepackage{amssymb}
\usepackage{url}
\usepackage{graphicx}
\usepackage{epstopdf}
\usepackage{comment}
\usepackage{color}

\def\eqref#1{(\ref{#1})}

\def\mR{\mathbb{R}}

\def\J1inv{\widehat{J}_0}

\begin{document}

\title{Affine Map Equivalence versus Critical Set Equivalence for Quadratic Maps of the Plane}

\author{Chia-Hsing Nien (Providence University, Taiwan), Bruce B. Peckham (University of Minnesota Duluth), Richard McGehee (University of Minnesota Twin Cities)}

\pagestyle{myheadings}
\thispagestyle{plain}
\markboth{\sc Nien, Peckham \& McGeghee}
{\sc Equivalences for quadratic maps of the plane}
\maketitle

\begin{center}
\today
\end{center} 

\renewcommand{\thefootnote}{\fnsymbol{footnote}}


\renewcommand{\thefootnote}{\fnsymbol{footnote}}
\footnotetext[2]{Department of Financial and Computational Mathematics, Providence University, Taichung City 43301 Taiwan}
\footnotetext[3]{Department of Mathematics and Statistics, University of Minnesota Duluth, Duluth, MN 55812, USA}
\footnotetext[4]{School of Mathematics, University of Minnesota, Minneapolis, MN 55455, USA} 

\renewcommand{\thefootnote}{\arabic{footnote}}

\begin{abstract}
In recent work \cite{NPM}, the authors enumerated a classification of quadratic maps of the plane according to their critical sets and images.
It is straightforward to show that quadratic maps which are affinely map equivalent are also equivalent in the critical set classification.
The question remained whether maps that are equivalent in the critical set classification are also affinely map equivalent.
This paper establishes a complete enumeration of the affine map equivalence classes.
As a consequence, the relationship between affine map equivalence and critical set equivalence is established.
In short, there are eighteen affine map equivalence classes.
Three pairs of those classes have critical sets and images that match, but each pair has some other geometric property, preserved by affine map equivalence, that does not match.
The other twelve affine map equivalence classes and the critical set equivalences are in one-to-one correspondence.

\end{abstract}


\pagestyle{myheadings}
\thispagestyle{plain}

\markboth{\sc C-H Nien, B B\ Peckham, and R P McGehee}{\sc Affine map equivalence vs critical set equivalence for quadratic maps}

\section{Introduction}\label{s-Introduction}

We are interested in studying the dynamics of quadratic maps of the real plane $Q:\mR^2\rightarrow \mR^2$.
The most general quadratic family has twelve coefficient parameters:
\begin{align}
\nonumber
Q(x,y)=(a_{20}x^2&+a_{11}xy+a_{02}y^2+a_{10}x+a_{01}y+a_{00},\\ &b_{20}x^2+b_{11}xy+b_{02}y^2+b_{10}x+b_{01}y+b_{00})
\label{eq-Q}
\end{align}
where at least one of the six quadratic coefficients is nonzero.

A complete description of the dynamics and bifurcations in this class would be an ambitious project since there is a huge variety of extremely complicated behaviors for quadratic planar maps.
This might seem surprising at first, but maybe not so much when one
realizes the effort expended in analyzing the dynamics of the Henon maps and the complex quadratic maps, both special subsets of quadratic maps of the plane.

In recent work \cite{NPM} (see also Nien's thesis \cite{Nienthesis}), we pursued a more modest goal, obtaining a complete classification, but based only on the critical set and its image, rather than a much finer dynamical classification based on topological conjugacy.
In that work, some reductions were made possible by the relatively straightforward observation that maps that are affinely map equivalent
necessarily have equivalent critical sets and images.
The question of whether maps that are {\it critical set equivalent} are also {\it affinely map equivalent} was left unanswered.
The current paper establishes this result.

\subsection{Affine map equivalence} 

We begin with the definitions of the two versions of equivalence.

\begin{definition}
Let $f:X \rightarrow X$ and $g:Y \rightarrow Y$ be $C^\infty$ maps.
Then $f$ is {\bf ($C^\infty$) map equivalent} to $g$ if there exist $C^\infty$ diffeomorphisms $h: X \rightarrow Y$ and $k: X \rightarrow Y$ such that $k \circ f=g \circ h$.
If, in addition, $h$ and $k$ are (nonsingular) affine maps, then $f$ and $g$ are said to be {\bf affinely map equivalent}.
Analogous properties for $h$ and $k$ can be used to define, for example, linear, $C^k$, analytic, or polynomial map equivalence.
\label{def-mapequiv}
\end{definition}

Note the following:

\begin{enumerate}
\item Our definition of map equivalence is taken from the singularity theory described in Chapter III of Golubitsky and Guillemin \cite{GG}; what we call {\it map equivalence} they just call {\it equivalence}. Delgado, et al. \cite{DGRRV} also use the same definition, but call it  {\it geometric map equivalence}.
\item A map that is affinely map equivalent to a quadratic map is itself quadratic.
\item A map that is linearly map equivalent to a homogeneous quadratic map is itself a homogeneous quadratic map.
\end{enumerate}

A primary result in this paper (see Theorem \ref{thm-ame}) is a complete enumeration of the 18 equivalence classes under affine map equivalence.
The proof is performed by a decision tree of coordinate changes.
Most of the coordinate changes are chosen to eliminate parameters, terminating with 18 parameter-free quadratic maps.
These are the 18 representatives, one from each equivalence class.
Once representatives are identified, it is a straightforward process
to compute their critical sets, and therefore the relationship between
affine map equivalence and critical set equivalence.

\subsection{Critical set equivalence}

For any differentiable map $f:\mR^2 \rightarrow \mR^2$ define the critical set $J_0$ by
\begin{equation}
        J_0 = J_0^f= \left\{(x,y)\in \mR^2 \,|\;
        \mathit{Df}(x,y) \mbox{\ is singular} \right\}=\{(x,y):\det(Df(x,y))=0\}.
\label{eq-J0}
\end{equation}
The image $J_1 = f(J_0)$ is called the {\it critical image}.

Since the partial derivatives of a quadratic function $Q$ from eq.~(\ref{eq-Q}) are linear functions of $x$ and $y$, the determinant of the two-by-two Jacobian derivative matrix $DQ(x,y)$ is a quadratic in $x$ and $y$:
\begin{align}
\nonumber
\det&\left( DQ(x,y)\right)=\begin{vmatrix}
2a_{20}x+a_{11}y+a_{10} &a_{11}x+2a_{02}y+a_{01}\\
 2b_{20}x+b_{11}y+b_{10} &b_{11}x+2b_{02}y+b_{01}
\end{vmatrix}\\
\nonumber
&=
(2a_{20}x+a_{11}y+a_{10})(b_{11}x+2b_{02}y+b_{01})-
(a_{11}x+2a_{02}y+a_{01})(2b_{20}x+b_{11}y+b_{10})\\
\nonumber
&\equiv Ax^2+Bxy+Cy^2+Dx+Ey+F\\
&=2X_{20:11}x^2+4X_{20:02}xy+2X_{11:02}y^2+(2X_{20:01}-X_{11:10})x+(X_{11:01}-2X_{02:10})y + X_{10:01}
\label{eq-detDF}
\end{align}
where
\begin{equation}
X_{ij:kl}=a_{ij}b_{kl}-a_{kl}b_{ij}=\begin{vmatrix}a_{ij}&a_{kl}\\b_{ij}&b_{kl}
\end{vmatrix}.
\end{equation}

\noindent
Thus $J_0$ is a conic section, possibly degenerate.
This provides our classification of $J_0$: $J_0^Q$ is one of the following: ellipse, hyperbola, parabola, point, pair of intersecting lines, two parallel lines (possibly coincident), single line, all of $\mR^2$, or the empty set. Note that the coincident lines could be logically grouped either with the parallel lines (as we do below) or with a single line.

\begin{theorem}
\label{thm-J0J1}
{\bf The $J_0$-$J_1$ Classification Theorem \cite{NPM}.} 
Let $F:\mR^2 \rightarrow \mR^2$ be a quadratic map of $\mR^2$, as in eq. (\ref{eq-Q}).
Then $J_0$ and $J_1$ take on one of the following forms:

\begin{enumerate}
\item $J_0$ is empty; $J_1$ is empty
\item $J_0$ is a point; $J_1$ is a point
\item $J_0$ is an ellipse;
$J_1$ is a closed curve with three cusp points
\item $J_0$ is a hyperbola; $J_1$ consists of two curves, one smooth, and the other smooth except for a single cusp; each curve is the image of one branch of the hyperbola
\item $J_0$ is a pair of intersecting lines; $J_1$ is one of the following:
\begin{enumerate}
\item the union of two rays emanating from the same point
\item the union of a ray and a parabola, with the endpoint of the ray and the vertex of the parabola coincident.
\end{enumerate}
\item $J_0$ is a parabola; $J_1$ a curve with a single cusp

\item $J_0$ is a pair of parallel lines
\begin{enumerate}
\item if the lines are distinct,
 $J_1$ is the union of a line and a point. One of the lines in $J_0$ maps onto the line in $J_1$ and the other line maps to the point in $J_1$.
 \item if the lines are coincident, then $J_1$ is a point which is the image of the whole line $J_0$.
 \end{enumerate}
\item $J_0$ is a single line; $J_1$ is one of the following.  
\begin{enumerate}
\item a point
\item a line
\item a parabola
\end{enumerate}
\item $J_0$ is all of $\mR^2$; $J_1$ is one of the following:
\begin{enumerate}
\item a line
\item a ray
\item a parabola
\end{enumerate}
\end{enumerate}
\end{theorem}

\begin{definition}
Let $f:\mR^2 \rightarrow \mR^2$ and $g:\mR^2 \rightarrow \mR^2$ be quadratic maps, each in the form of eq.~(\ref{eq-Q}).
Then $f$ is {\bf critical set equivalent} or ``{\bf $J_0$-$J_1$ equivalent}'' to $g$ if 
they have the same $J_0$-$J_1$ classification as in Theorem \ref{thm-J0J1}.
\label{def-J0J1equiv}
\end{definition}

By Theorem \ref{thm-J0J1}, there are fifteen equivalence classes under critical set equivalence.
Note that we have chosen to distinguish in our classification a double line from a single line.
These cases can be distinguished algebraically, for example, by $\det(DF(x,y)$ being $x$ versus $x^2$.

\subsection{Comparison}

The fact that two quadratic maps being affinely map equivalent implies they are critical set equivalent is a straightforward application of the following observation.

\begin{lemma}
\label{lemma-JfJg}
Let $f$ and $g$ be quadratic maps of the plane.  If $f$ and $g$ are affinely map equivalent, then they are critical set equivalent.
\end{lemma}
\begin{proof}
Let $J_0^f$ be the critical set for $f$ and $J_1^f$ its image under $f$.  Similarly, define $J_1^g = g(J_0^g)$. 
Assume $f$ and $g$ are map equivalent.
Then their critical sets are affine images of each other.
More specifically, if $g=k \circ f \circ h^{-1}$,
then $J_0^g = h(J_0^f)$, and $J_1^g = k(J_1^f)$.
The result now follows immediately because invertible affine maps preserve all the geometric properties used in the classification of $J_0$ and $J_1$: ellipses, hyperbolas, parabolas, lines, points, rays, empty sets, intersection/nonintersection of curves, and cusps on curves.
\end{proof}

Lemma \ref{lemma-JfJg} was utilized in \cite{NPM} to help prove the $J_0$-$J_1$ classification theorem, but the proofs in that paper are vastly simplified by the use of the Affine Map Equivalence Theorem (\ref{thm-ame}) below.
Similarly, \cite{DGRRV} uses Lemma \ref{lemma-JfJg} to help prove that the two generic cases, where $J_0$ is an ellipse, or hyperbola, each belong to a single equivalence class under {\it $C^\infty$ map equivalence}.
Our result strengthens theirs (each of these two cases belongs to a single equivalence class under the stronger condition of {\it affine map equivalence}), and 
their proofs are also vastly simplified and shortened by using the Affine Map Equivalence Theorem. 

\subsection{Organization}\label{ss-organization}
The rest of the paper is organized as follows.
In section \ref{s-AME}, we state Theorem \ref{thm-ame}, the affine map equivalence theorem for quadratic maps of the plane, and outline its proof.
There turn out to be exactly eighteen affine map equivlaence classes, and fifteen critical set equivalence classes.
This leads to a partial converse of Lemma \ref{lemma-JfJg}:
maps which are critical set equivalent are affine map equivalent, except for three pairs of maps which are critical set equivalent, but not affinely map equivalent.
In section \ref{s-consequences}, we state several corollaries and consequences of the affine map equivalence theorem, and discuss related results. In section \ref{s-discussion} we discuss alternate versions of map equivalence, and point out some interesting relationships to other areas.
The details of the proofs of the lemmas used to prove Theorem \ref{thm-ame} are presented in section \ref{s-lemma-pf}.

\section{The Affine map equivalence theorem}
\label{s-AME}

\begin{theorem}
\label{thm-ame}
Let $f$ be a quadratic map of the plane as in eq.~(\ref{eq-Q}).
Then $f$ is affinely map equivalent to exactly one of the 18 maps in table \ref{table-ame}. The description of the critical set $J_0^f$ and image $J_1^f$ is the same for any map in the same affine map equivalence class, as are the cardinalities of preimages.
\end{theorem}

\begin{table}[ht]
\begin{tabular}{|c|c|c|c|c|}
\hline \hline

Label&Normal form&$J_0$&$J_1$&Preimage\\
&representative&&&cardinalities\\
\hline \hline
$E_1$&$(x^2-y^2+x, xy)$ &ellipse&3-cusped curve& $2,3,4$\\
\hline
$E_2$&$(x^2-y^2, xy)$&point&point&1,2\\
\hline \hline
$H_1$&$(x^2+y^2+x, xy)$&hyperbola&two disj. curves,&$0,1,2,3,4$\\
&&& one w/ cusp&\\
\hline
$H_2$&$(x^2+y^2+x, xy+ \frac{1}{2}x)$&$2$ intersecting&parab. $\bigcup$ ray &$0,1,2,3,4$\\
&&lines&&\\
\hline
$H_3$&$(x^2+y^2, xy)$&$2$ intersecting &ray $\bigcup$ ray&$0,1,2,4$\\
&&lines&&\\
\hline \hline
$P_1$&$(x^2+y, xy)$&parabola&curve with cusp&$1,2,3$\\
\hline
$P_2$&$(x^2, xy+y)$&$2$ parallel&line $\bigcup$ point&$0,1,2,\infty$\\
&&lines&&\\
\hline
$P_3$&$(x^2, xy)$&$2$ coincident lines&point&$0,2,\infty$\\
\hline \hline
$D^E_1$&$(x^2-y^2,y)$&line&parabola *&$0,1,2 \; \dagger$\\
\hline
$D^E_2$&$(x^2-y^2, x+y)$&line&point&$0,1,\infty$\\
\hline
$D^E_3$&$(x^2-y^2, 0)$&$\mR^2$&line **&$0,\infty$\\
\hline \hline
$D^H_1$&$(x^2+y^2, y)$&line&parabola *&$0,1,2 \; \dagger$\\
\hline
$D^H_2$&$(x^2+y^2, 0)$&$\mR^2$&ray ***&$0,1,\infty$\\
\hline \hline
$D^P_1$&$(x^2+y, x)$&$\phi$&$\phi$&$1$\\
\hline
$D^P_2$&$(x^2,y)$&line&line&$0,1,2 \; \dagger$\\
\hline
$D^P_3$&$(x^2+y, 0)$&$\mR^2$&line **&$0,\infty \; \dagger \dagger$\\
\hline
$D^P_4$&$(x^2, x)$&$\mR^2$&parabola&$0,\infty \; \dagger \dagger$\\
\hline
$D^P_5$&$(x^2, 0)$&$\mR^2$&ray ***&$0,\infty$\\
\hline \hline
\end{tabular}
\caption{Representatives of the 18 cases in Theorem \ref{thm-ame} with their critical sets, $J_0$ and $J_1$, 
and
cardinalities of preimage sets.
Double lines separate groups with common quadratic terms.
A more detailed description of the $J_1$ sets is included in the statement of Theorem \ref{thm-J0J1}.
Asterisk groups mark classes that are critical set 
equivalent, but not affinely map equivalent.
Dagger groups mark classes that are homeomorphically (and polynomially) map equivalent, but not affinely map equivalent. See Sec. \ref{ss-other-equiv}.
Notation for the label is chosen to match the generic critical set $J_0$ in each grouping: $E$: ellipse, $H$: hyperbola, $P$: parabola. $D$ stands for degenerate, with the homogeneous part of the second component zero, and with the superscript chosen to match the algebraic form of the homogeneous part of the first component.
Subscripts enumerate the classes in each group.}

\label{table-ame}
\end{table}


\subsection{Phase space examples}
\label{s-phase}


In Figure \ref{fig-18phase} we provide illustrations of images of a disk for an example in each of our classification cases.  See table \ref{table-ame} for further information about the properties of these cases.

\begin{figure}[htbp]

\begin{tabular}{|c|c|c|}
\hline
$E_1$&$E_2$&$H_1$\\

\includegraphics[width=.3\textwidth]{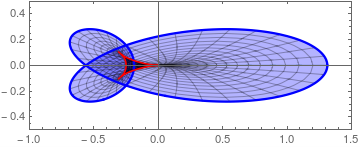}
&\includegraphics[width=.15\textwidth]{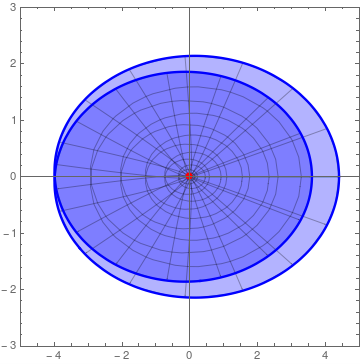}
&\includegraphics[width=.25\textwidth]{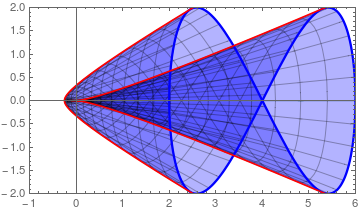}\\
\hline

$H_2$&$H_3$&$P_1$\\

\includegraphics[width=.25\textwidth]{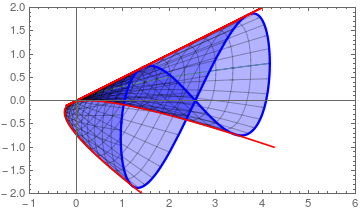}
&\includegraphics[width=.25\textwidth]{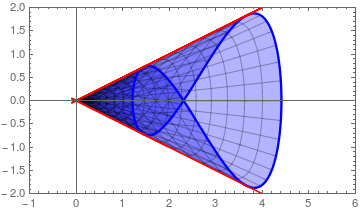}
&\includegraphics[width=.25\textwidth]{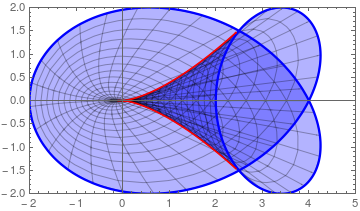}\\
\hline

$P_2$&$P_3$&$D^E_1$\\

\includegraphics[width=.13\textwidth]{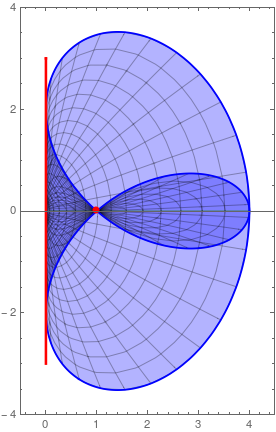}
&\includegraphics[width=.25\textwidth]{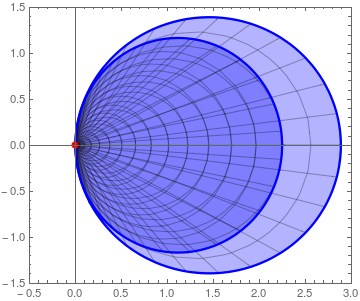}
&\includegraphics[width=.25\textwidth]{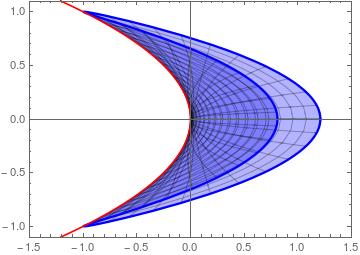}\\
\hline

$D^E_2$&$D^E_3$&$D^H_1$\\

\includegraphics[width=.2\textwidth]{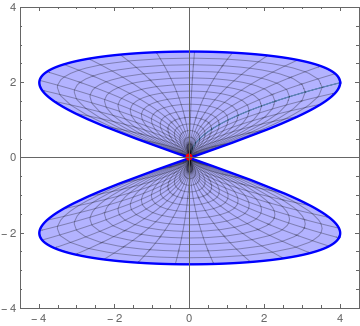}
&\includegraphics[width=.25\textwidth]{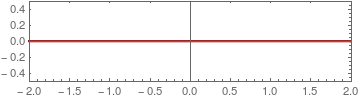}
&\includegraphics[width=.17\textwidth]{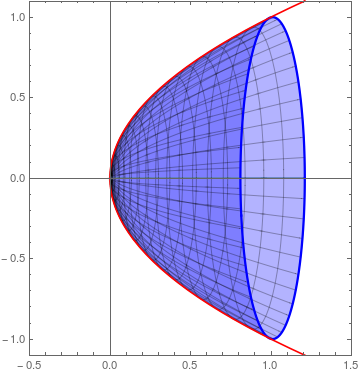}\\
\hline

$D^P_1$&$D^P_2$&$D^P_4$\\

\includegraphics[width=.25\textwidth]{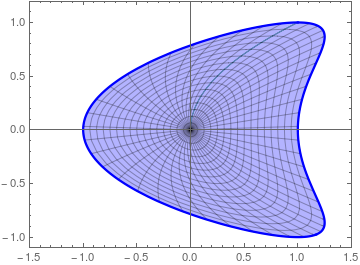}
&\includegraphics[width=.18\textwidth]{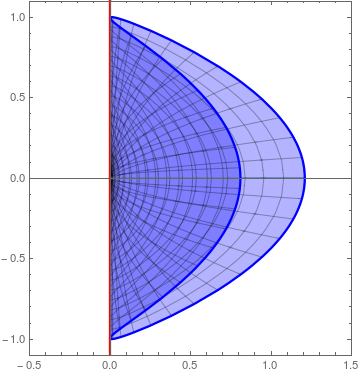}
&\includegraphics[width=.16\textwidth]{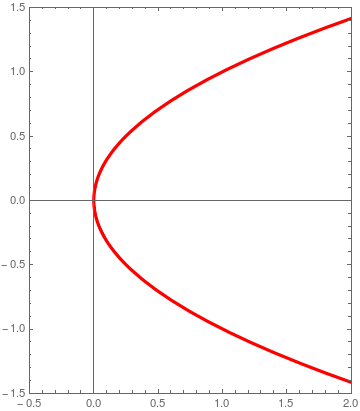}\\
\hline

$D^H_2$&$D^P_3$&$D^P_5$\\

\includegraphics[width=.25\textwidth]{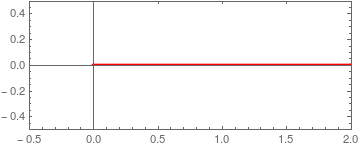}
&\includegraphics[width=.25\textwidth]{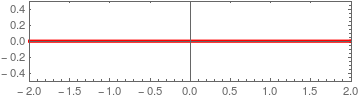}
&\includegraphics[width=.25\textwidth]{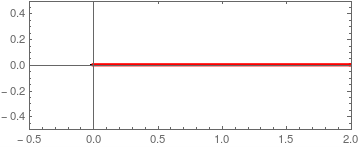}\\
\hline
\end{tabular}
\caption{Images of disks for the $18$ affine map equivalence representatives. Centers are offset from the origin for better viewing for $E_2, H_3, P_3, D^E_1, D^H_1$, and $D^P_2$. Critical values, $J_1$, are in red.  Refer to Table \ref{table-ame} for further description.}
\label{fig-18phase} 
\end{figure}

\subsection{Proof of Theorem \ref{thm-ame}}
\label{ss-ame-proof}

\begin{proof}
The proof of the $18$ equivalence classes is effected by sequences of coordinate changes, both in the domain and range.
The algorithm for the coordinate changes terminates in exactly one of the 18 normal forms.
In other words, given any quadratic map $f$, we find nonsingular affine changes of coordinates $h$ (in the domain) and $k$ (in the range) such that $k \circ f \circ h^{-1}$ is one of the 18 representatives in table \ref{table-ame}.
That no pair of the 18 representatives is equivalent could be established algebraically by 
showing there is no invertible affine pair of maps $(h,k)$ such that $k Q_1=Q_2h$ for pair of maps $Q_1$ and $Q_2$, where $Q_1$ and $Q_2$ are different normal form representatives from Table \ref{table-ame}.
It is much easier (and more elegant), however, to find invariants of affine map equivalence, that
do not match between pairs of normal form representatives.
Lemma \ref{lemma-JfJg} is an example of properties ($J_0$ and $J_1$), that must match in order for the two maps to be affinely map equivalent.
This takes care of all pairs except in the cases where the corresponding $J_0$ and $J_1$
match: \{$D^E_1$, $D^H_1$\}, \{$D^E_3$, $D^P_3$\}, and \{$D^H_2$, $D^P_5$\}.
These three pairs turn out to be inequivalent because some other invariant of affine map equivalence (preimages or convexity) does not match.
More details are in Lemma \ref{lemma-inequivalent}.

Note that the homogeneous quadratic terms of $k \circ Q \circ h^{-1}$
are unaffected by the linear and constant terms of $Q$.
More formally, for any quadratic map $Q$ as in eq.~(\ref{eq-Q}), if we define the projection to the homogeneous quadratic terms, $\pi:Q \mapsto \pi Q $ by
\begin{align}
\pi Q(x,y)=(a_{20}x^2&+a_{11}xy+a_{02}y^2, b_{20}x^2+b_{11}xy+b_{02}y^2)
\label{eq-pQ}
\end{align}
then $\pi( k \circ Q \circ h^{-1})=\pi(k \circ \pi(Q) \circ h^{-1})$.
Additionally, if $h$ and $k$ are linear invertible maps, then 
$\pi(k \circ \pi(Q) \circ h^{-1})=k \circ \pi(Q) \circ h^{-1}$.
So, we deal with quadratic terms first, in Lemma \ref{lemma-homogeneous}, to identify six distinct homogeneous normal forms, then add back the linear and constant terms.
The subsequent lemmas \ref{lemma-elliptic} through \ref{lemma-deg-parabolic} each deal with one of these six homogeneous forms.
Note that the six homogeneous forms in Lemma \ref{lemma-homogeneous} each appear themselves in the list of 18 equivalence class representatives in Theorem \ref{thm-ame}.

\begin{lemma}
\label{lemma-homogeneous}
Let $Q_0$ be any homogeneous map of the form of eq.~(\ref{eq-pQ}).
Then $Q_0$ is affinely map equivalent to a quadratic map $G$, where $\pi G$ takes on one of the following six forms: $E_2=(x^2-y^2, xy)$, $H_3=(x^2+y^2, xy)$, $P_3=(x^2, xy)$, $D_3^E=(x^2-y^2, 0)$, $D_2^H=(x^2+y^2,0)$, or $D_5^P=(x^2,0)$.
\end{lemma}

\begin{lemma}
\label{lemma-elliptic}
Let $Q$ be a quadratic map with homogeneous part affinely map equivalent to 
$(x^2-y^2, xy)$.  Then $Q$ is affinely map equivalent to one of the following two quadratic maps: $E_1=(x^2-y^2+x, xy)$, or $E_2=(x^2-y^2,xy)$.
\end{lemma}

\begin{lemma}
\label{lemma-hyperbolic}
Let $Q$ be a quadratic map with homogeneous part affinely map equivalent to 
$(x^2+y^2, xy)$.  Then $Q$ is affinely map equivalent to one of the following three quadratic maps: $H_1=(x^2+y^2+x,xy)$, $H_2=(x^2+y^2+x, xy+\frac{1}{2}x)$ or $H_3=(x^2+y^2,xy)$.
\end{lemma}

\begin{lemma}
\label{lemma-parabolic}
Let $Q$ be a quadratic map with homogeneous part affinely map equivalent to 
$(x^2, xy)$.  Then $Q$ is affinely map equivalent to one of the following three quadratic maps: $P_1=(x^2+y,xy)$, $P_2=(x^2, xy+y)$ or $P_3=(x^2,xy)$.
\end{lemma}

\begin{lemma}
\label{lemma-deg-elliptic}
Let $Q$ be a quadratic map with homogeneous part equivalent to 
$(x^2-y^2, 0)$.  Then $Q$ is affinely map affinely map equivalent to one of the following three quadratic maps: $D^E_1=(x^2-y^2,y)$, $D^E_2=(x^2-y^2, x+y)$ or $D^E_3=(x^2-y^2,0)$.
\end{lemma}

\begin{lemma}
\label{lemma-deg-hyperbolic}
Let $Q$ be a quadratic map with homogeneous part equivalent to 
$(x^2+y^2, 0)$.  Then $Q$ is affinely map affinely map equivalent to one of the following two quadratic maps: $D^H_1=(x^2+y^2, y)$ or $D^H_2=(x^2+y^2,0)$.
\end{lemma}

\begin{lemma}
\label{lemma-deg-parabolic}
Let $Q$ be a quadratic map with homogeneous part affinely map equivalent to 
$(x^2, 0)$.  Then $Q$ is affinely map equivalent to one of the following five quadratic maps: $D^P_1=(x^2+y,x)$, $D^P_2=(x^2,y)$,$D^P_3=(x^2+y,0)$, $D^P_4=(x^2, x)$ or $D^P_5=(x^2,0)$.
\end{lemma}

\begin{lemma}
\label{lemma-inequivalent}
The following sets of quadratic maps are inequivalent under affine map equivalence.  A sketch of the proof in each case is listed with the result. 
Further details are provided 
in Lemma \ref{lemma-top-invariants}.

\begin{enumerate}
\item $D^E_1=(x^2 - y^2,y)$ and $D^H_1=(x^2+y^2,y)$.  Both have a line for $J_0$ and a parabola for $J_1$, but the range of $D^H_1$ is convex, while the range of $D^E_1$ is non-convex. 

\item $D^E_3=(x^2-y^2, 0)$ and $D^P_3=(x^2+y, 0)$.   Both have $\mR^2$ for $J_0$ and a line for $J_1$, but preimages of points in $J_1$ for $D^E_3$ are hyperbolas, while preimages of points in $J_1$ for $D^P_3$ are parabolas.   A parabola cannot be even topologically equivalent to a hyperbola since the hyperbola has two branches and is therefore a disconnected set.

\item $D^H_2=(x^2+y^2, 0)$ and $D^P_5=(x^2, 0)$.   Both have $\mR^2$ for $J_0$ and a ray for $J_1$, but preimages of points in $J_1$ for $D^H_2$ are all circles (except for the origin which has only itself as a preimage, and preimages of points in $J_1$ for $D^E_5$ are all pairs of lines (except for the origin which has only the $y$-axis as its preimage).  Circles and pairs of lines are not mapped to each other by affine maps, so these two maps are not affinely map equivalent.  They are not even topologically maps equivalent since circles and points are compact, and lines are not compact.

\item $D^E_2=(x^2 - y^2,x+y)$ and $P_3=(x^2,xy)$.   Both have a line for $J_0$ and a point for $J_1$, but $D^E_3$ has a single line for $J_0$, while $P_3$ has a double line for $J_0$. 
In addition (if one needs a second violation of affine map equivalence), 
all points not on $J_0$ for $D^E_3$ have unique preimages, while no points in $\mR^2$ have unique preimages for $P_3$.

\end{enumerate}
\end{lemma}

Lemmas \ref{lemma-elliptic} through \ref{lemma-deg-parabolic} guarantee that there are no more than $18$ affine map equivalence classes.
Lemmas \ref{lemma-JfJg} and \ref{lemma-inequivalent}
show that no two of the $18$ representatives listed in Theorem \ref{thm-ame} are equivalent, and thus that there are exactly $18$ equivalence classes. 
The form of $J_0$ and $J_1$ is easily computed since it needs to be done for only these 18 individual representatives.
That the same description of $J_0$ and $J_1$ holds for all maps in the same affine map equivalence class comes from Lemma \ref{lemma-JfJg}.
Similarly, the cardinality of preimages is easily determined for each of the $18$ individual representatives.
Lemma \ref{lemma-top-invariants} guarantees the same preimage cardinalities for all maps in the same affine map equivalence class.
\end{proof}



\section{Consequences}
\label{s-consequences}

Several consequences follow immediately from Theorem \ref{thm-ame} and Theorem \ref{thm-J0J1}.

\subsection{Relationship between affine map equivalence and critical set equivalence}
The equivalence classes for these two classifications are very close to being the same.
Only the first three pairs in Lemma \ref{lemma-inequivalent} are 
critical set equivalent, but not affinely map equivalent.
Identifying these three pairs takes us from affine map equivalence to critical set equivalence.
If one thinks geometrically, affine map equivalence of a pair of maps requires a relationship on the whole domain and range of the maps, while critical set equivalence requires a relationship only on the critical sets $J_0$ and $J_1$.
At first glance, it is seems surprising that these two classifications should be so close.
That is, it is surprising that knowledge of only critical sets forces a relationship on the whole domain and range.
This seems somewhat plausible if we rephrase the relationship as: if we know what kind of folds a surface has, then we can deduce geometric properties of the whole surface.
Of course this intuitive explanation is incomplete because not all of the critical sets correspond to folds.

\subsection{Proofs in \cite{NPM}}
\label{ss-shortCSEproofs}
At the time we (the coauthors of this paper) proved the $J_0$-$J_1$
classification theorem in \cite{NPM}, we had not yet established the Affine Map Equivalence Theorem.  
We did not even know whether the set of affine map equivalence classes was finite.
Consequently, we expended a lot of effort in \cite{NPM} to prove, for example, that all quadratics maps with critical set an ellipse had as $J_1$ a closed curve with three cusps.
This required carrying many parameters along through substantial computations.
Now, with the affine map equivalence theorem, we merely need to establish this result for the map $E_1$ (or any other single map in its equivalence class), and then note that the three cusps are preserved by affine maps.
This saves many pages of calculations that we had previously thought to be necessary.

\subsection{Proofs of geometric equivalence in \cite{DGRRV}}
\label{ss-geomequiv}

In \cite{DGRRV} (see also the earlier preprint \cite{GRRV}), the authors claim that all planar quadratic maps with ellipses for $J_0$ are $C^\infty$ map equivalent to each other.
Theorem \ref{thm-ame} is even stronger: all such maps are affinely map equivalent to $E_1$.
The extensive proofs in \cite{DGRRV} to arrive at this conclusion are no longer necessary.
For example, the authors also prove, but through quite different calculations from \cite{NPM}, that any quadratic map with an ellipse for $J_0$ has image $J_1$ a closed curve with three cusp points.
In addition, they prove that all such maps send $J_0$ injectively onto $J_1$.
We need only establish this result for one 
map in the equivalence class, and the affine equivalence will ensure that these properties will hold for all maps in the equivalence class.
For example, the map $(x^2-y^2+2x, 2xy -2y)$ (this is the same as $z^2+2\overline{z}$ in complex coordinates, and is in the affine map equivalence class of $E_1$)
is easily shown to have the unit circle as its $J_0$, and
image parametrized (bijectively) by $\theta \in [0,2 \pi)$ as $e^{2i \theta} + e^{-i\theta}$.
This last expression is a standard parametrization of the deltoid, a three-cusped hypocycloid.
Thus, all maps in class $E_1$ have a $J_1$ that not only is a three-cusped curve, but is additionally an affine image of a deltoid.

Moreover, the authors in \cite{DGRRV} claim that there is a finite set of $C^\infty$ map equivalence classes, although, for proof of this claim, they refer to previous work \cite{BGVR}, which focuses on other results and does not explicitly enumerate these equivalence classes.
Their claim of finiteness of equivalence classes for $C^\infty$ map equivalence becomes a straightforward corollary of our stronger statement of the Affine Map Equivalence Theorem.
A complete, explicit enumeration of the fifteen $C^\infty$ map equivalence classes is given in Theorem \ref{thm-OME} below.

\section{Discussion}
\label{s-discussion}

\subsection{Topological  and other map equivalences}
\label{ss-other-equiv}

The comparisons in subsection \ref{ss-geomequiv} between $C^\infty$ map equivalence and affine map equivalence 
for planar quadratic maps suggest looking at map equivalences via coordinate changes that are $C^0$ (homeomorphisms), $C^k$, $C^\infty$ $C^\omega$ (analytic) or even polynomial.
Of course, not all such coordinate changes preserve the class of quadratic maps of the plane.
This is an argument that affine map equivalence is the most natural
map equivalence to consider.
Nevertheless, we make the following observations.
Since affine maps fall into all of the above categories, then the $18$ affine map classes can only collapse when moving from affine to any other type of map equivalence, such as in the following lemma.

\begin{lemma}
\label{lemma-groups}
The maps in each of the following sets are affine map inequivalent, but 
polynomially map equivalent.

\begin{enumerate}
\item $\{D^E_1=(x^2-y^2, y), D_1^H=(x^2+y^2, y), D_2^P=(x^2,y)\}$
\item $\{ D^P_3=(x^2 + y,0), D^P_4=(x^2,x) \}$.
\end{enumerate}
\end{lemma}
\begin{proof}
\begin{enumerate}
\item Let $k_1(x,y)=(x+y^2, y)$, and $k_2(x,y)=(x-y^2, y)$.
Then $D_2^P=k_1 D^E_1= k_2 D_1^H$.
\item Let $h(x,y)=(y, x-y^2)$, and $k(x,y)=(y,x-y^2)$.
Then $D^P_3 h = k D^P_4=(x,0)$.
\end{enumerate}
\end{proof}

This lemma provides a collapsing of equivalence classes from affine map equivalence to polynomial map equivalence, and therefore to any $C^k$ map equivalence.  In the rest of this section, we argue that there is no further collapsing because of invariants of the coarsest class we consider, $C^0$ map equivalence.

We have already noted in Lemma \ref{lemma-JfJg} that
if $f$ and $g$ are map equivalent via $g=k \circ f \circ h^{-1}$, and $h$ and $k$ are both at least $C^1$, then 
$J^g_0=h(J_0^f)$ and $J^g_1=k(J^f_1)$.
There are useful topological properties which are also preserved.

 \begin{lemma}
 \label{lemma-top-invariants}
 Let $f$ and $g$ be maps from a space $X$ to itself.
 If $g=k \circ f \circ h^{-1}$, where $h$ and $k$ are homeomorphisms of $X$, then preimages and images of corresponding sets are homeomorphic.  More specifically,
\begin{enumerate}
\item Preimages of corresponding points are homeomorphic
\item The range of $f$ and $g$ are homeomorphic,

\item If, in addition to the theorem hypotheses, $h$ and $k$ are affine (and nonsingular), then the range of $f$ and $g$ must have the same convexity.

\end{enumerate}
\end{lemma}

\begin{proof}
\begin{enumerate}
\item  If $k(y)=v$, 
then $h(f^{-1}(y)) = g^{-1}(v)$.
This last equality is easily established by showing that an element of either side is also an element of the other side.
\item  Since $h$ is onto, $k(g(X))=f(h(X))=f(X)$.
\item The ranges are now not only homeomorphic, but mapped to each other by the affine map $k$.
Affine maps take line segments to line segments, and therefore convex sets to convex sets.

\end{enumerate}
\end{proof}

Note that homeomorphic preimage sets requires the same cardinality of preimage sets, so this property alone immediately rules out topological map equivalence for many of the $18$ representative maps in the Affine Map Equivalence Theorem.
See the rightmost column in Table \ref{table-ame} where the cardinality of preimages for each representative (and therefore each map in its affine map equivalence class) is listed.
These cardinalities are easily determined algebraically from the representative example, and they can be verified by studying the figures in Fig.~\ref{fig-18phase}.
For example, preimages of points for $E_1$ can only have cardinalities of $2,3$ or $4$.
No other of the $18$ representatives has this set of preimage cardinalities, so $E_1$ is inequivalent to any of the other affine map classes, even under $C^0$ map equivalence.

Further, maps with preimage cardinalities of $\infty$ can be distinguished by other topological properties such as compactness or connectedness.  
For example, $D^E_3$, $D^H_2$, and $D^P_3$ are distintguished in topological map equivalence since preimages of points in the respective ranges can be hyperbolas (noncompact, nonconnected), circles (compact, connected), and parabolas (noncompact, connected).
Of course, $D^H_2$ was already distinguished since it has a point, the origin, with only one preimage.

In addition, $C^0$ (homeomorphic) map equivalence preserves ``sets of merging preimages'', as defined in pioneering work on planar endomorphisms by Mira and coworkers \cite{GM, FGKMM}.
Often these sets of merging preimages coincide with our definition of $J_1$, so sets with different topological descriptions of $J_1$ will be in distinct $C^0$ map equivalence classes.

Finally, combining Lemma \ref{lemma-top-invariants} with \ref{lemma-JfJg}, cardinalities of preimages restricted to critical sets are preserved.
In particular, maps topologically map equivalent to $E_1, E_2, H_1, P_1, D^E_1, D^H_1$, and $D^P_2$ are all injective on their critical sets.


Lemma \ref{lemma-top-invariants} leads to the following classification theorem.

\begin{theorem}
\label{thm-OME}
There are exactly fifteen map equivalence classes for quadratic maps of the plane where the bijective coordinate changes are restricted to any one of the following groups: 
$C^0, C^1, ..., C^k, C^\infty, C^\omega$, or polynomial.
The equivalence classes agree with the $18$ affine map equivalence classes in Table \ref{table-ame} except for the identification of the two groups in Lemma
\ref{lemma-groups}: $\{D^E_1, D^H_1, D^P_2\}$ and $\{D^P_3, D^P_4\}$.
\end{theorem}

Note that these fifteen equivalence classes are not the same as the fifteen critical set equivalence classes.
The eight nondegenerate affine map equivalence classes, $E_i, H_i, P_i$ are all distinct under any version of map equivalence, from
$C^0$ through affine, and critical set equivalence.
Differences occur only in the degenerate classes.

\subsection{Algebraic terminology}
\label{ss-algebraic}
From a purely algebraic point of view, we can view map equivalences as orbits of group actions on sets.
If we let the sets be all quadratic maps, as in eq.~(\ref{eq-Q}),
and the group be pairs of invertible affine maps $(h,k)$ of $\mR^2$, and the action of the group on the set be $Q \rightarrow k \circ Q \circ h^{-1}$, then the affine map equivalence classes are the $18$ distinct orbits of elements of the quadratic maps under this group action.

\subsection{Dynamics}
\label{ss-dynamics}
Dynamical equivalence via topological conjugacy is a much finer classification since the changes of coordinates must be the same on both the domain and range ($h = k$).
But the map equivalence classification suggests a starting point for further dynamical studies.
The Henon maps \cite{Henon}, for example, are a study of maps in our class $D^P_1$, and \cite{Nien98}, for example is the beginning of a study of maps in our classes $E_1$ and $E_2$.
The complex quadratic family $z^2+c$ is a special subset of maps in our class $E_2$.

\subsection{Jacobian Conjecture}
The Jacobian Conjecture states that polynomial maps with nonzero constant Jacobian determinant are invertible with polynomial inverses.
The affine map classification provides an easy proof of a special case of this Conjecture.

\begin{lemma}
If $Q:\mR^2 \rightarrow \mR^2$ is a quadratic map with $\det(DQ)$ a nonzero constant,
Then $Q$ has a quadratic polynomial inverse.
\end{lemma}

\begin{proof}
Since $J_0^Q = \phi$. Hence by Theorem \ref{thm-ame}, we have
$Q = k (x^2 +y, x) h^{-1}$, for some nonsingular affine maps $h$ and $k$.
Since  $(x^2+y, x)^{-1} = (y, x - y^2)$,
we have $Q^{-1} = h (y, x - y^2) k^{-1}$ which is a quadratic polynomial map.
\end{proof}

\section{Proofs of the Lemmas}
\label{s-lemma-pf}

Before launching into the proofs, we point out some preliminary observations and notational conventions that we use.
Recall the most general planar quadratic map from eq.~(\ref{eq-Q}).
$Q(x,y)=(a_{20}x^2+a_{11}xy+a_{02}y^2+a_{10}x+a_{01}y+a_{00}, b_{20}x^2+b_{11}xy+b_{02}y^2+b_{10}x+b_{01}y+b_{00})$.

\begin{itemize}

\item The constant terms $a_{00}$ and $b_{00}$ can be eliminated by the translation
$k(x, y) = (x - a_{00},\; y - b_{00})$ at any step.
To save space (and work), we will not include them in any step of the computation.
That is, when we say $Q_i h, k Q_i $ or $kQ_i h$ is ``of the form'' $Q_j$, we might need to compose on the left with this additional translation before we obtain equality to $Q_j$.

 \item We reuse the coefficient notation $a_{ij}$ and $b_{ij}$ after all coordinate changes.  Thus, if $Q_i \rightarrow Q_j$, prior to the coordinate change, $a_{20}$ is the $x^2$ coefficient for the first component of $Q_i$, while after the coordinate change,
$a_{20}$ is the $x^2$ coefficient for the first component of $Q_j$.
We also reuse the $Q_i$ notation in the proof of each lemma.
The notation for the eighteen normal forms, however, is fixed throughout the paper.

\item When $h$ and $k$ (and therefore their inverses) are both translations, the quadratic terms of $kFh$ remain the same as $F$. 
Translations will often be the last steps in our coordinate changes.

\item For emphasis, we box maps in the organizational diagrams that are in one of the eighteen normal forms in Theorem \ref{thm-ame}.

\item We often abuse correct formal notation and write $Q$ as a shorthand for $Q(x,y)$.

\end{itemize}
%

\subsection{Proof of Lemma \ref{lemma-homogeneous}}

The organization of the proof is given by the schematic diagram of coordinate changes in table \ref{table-homogeneous}.

\begin{table}[ht]
\label{table-homogeneous}
\begin{tabular}{ccccccccccc} 
$Q_0$&$\rightarrow$&$Q_1$&$\rightarrow$ &$Q_2$ &$\rightarrow$ &$Q_3$ &$\rightarrow$ &$Q_4$&$\rightarrow$ &\framebox[.8cm][l]{$E_2$}\\
& &&&&$\swarrow$&&&$\downarrow$&$\searrow$\\
&&&&$Q_5$ && &&\framebox[.8cm][l]{$P_{3}$} & &\framebox[.8cm][l]{$H_3$}\\
&&&$\swarrow$&$\downarrow$&$\searrow$\\
&&\framebox[.9cm][l]{
$D^E_{3}$}&&\framebox[.9cm][l]{$D^P_5$} && \framebox[.8cm][l]{$D^H_2$}& &\\

\end{tabular}
\caption{Organization of the coordinate changes to prove Lemma \ref{lemma-homogeneous}.  The six boxed forms are ``normal forms'' that appear in the list of 18 equivalence class representatives in Table \ref{table-ame}.}
\end{table}

As indicated in subsection \ref{ss-ame-proof}, this first lemma will deal with only the homogeneous quadratic terms.
For this lemma only, we will amend our terminology 
``$k \circ Q_i \circ h$ is of the form $Q_j$'' to mean that the form of the homogeneous coefficients match; 
$k \circ Q_i \circ h$ may have nonzero linear and constant terms, but $Q_j$ will only have homogeneous degree two terms, and can still be ``of the form" of $k \circ Q_i \circ h$. 

Begin with $Q_0(x,y)$ the most general homogeneous quadratic map.

\begin{itemize}
\item $Q_0(x,y) = (a_{20} x^2 + a_{11} xy + a_{02} y^2, b_{20} x^2 + b_{11} xy + b_{02} y^2)$.
\item $Q_0 \rightarrow Q_1$, where $Q_1$ 
has coefficient of the $xy$ term in $\det(DQ_1)(x, y)$, $X_{20:02} = a_{20}b_{02} - a_{02}b_{20}$, equal to $0$. 
Finding such a $Q_1$ is always possible because the $xy$ term in any conic section can be eliminated by a rotation.
Letting $R_\theta$ be the rotation by the appropriate $\theta$, $Q_0 \rightarrow Q_1=R_\theta Q_0 R_\theta^{-1}$, we can ensure that its $X_{20:02} = 0$.

\item $Q_1 \rightarrow Q_2$, where
$Q_2$ has its $X_{20:02}=0$ and its $a_{20}\ne 0$.
If $a_{20}$ for $Q_1$ is already nonzero, then let $Q_2=Q_1$.
If $a_{20}=0$, but any of $a_{02}, b_{20}, b_{02}$ is nonzero, then let $S(x,y)=(y,x)$.
One of $SQ_1, Q_1S$, or $SQ_1S$ has its $a_{20}$ nonzero.
If all four of
$a_{20}, a_{02}, b_{20}, b_{02}$ are zero, then at least one of 
$a_{11}$ or $b_{11}$ is nonzero, or $Q_1$ would not be a quadratic map.
Letting $A=a_{11}$, $B=b_{11}$, and $h(x,y)=(x+y, x-y)$, $Q_1 h = (A x^2 - A y^2, B x^2 - B y^2)$.
If $A \ne 0$, let $Q_2=Q_1 h$.
If $A = 0$, $Q_1 h = (0, B x^2 - B y^2)$, with $B \ne 0$.
In this case,
$Q_2(x,y)=S Q_1 h(x,y) = (B x^2 - B y^2,0)$.
So in all possible cases,
$Q_2$ can be produced so that its $a_{20}\ne 0$;
in all cases $Q_2$ will still have its
$X_{20:02}=0$.

\item $Q_2 \rightarrow Q_3= (x^2 + a_{11} x y + a_{02} y^2,\; b_{11} x y)$.
Since $X_{20:02} = 0$, and $a_{20} \ne 0$, define $r = b_{20}/a_{20}$,
then $b_{02} = (a_{02}b_{20})/a_{20} = ra_{02}$. So $Q_2$ is a map of the form:
$Q_2(x, y) = (a_{20} x^2 + a_{11} x y + a_{02} y^2,\; ra_{20} x^2 + b_{11} x y + ra_{02} y^2)$.

We introduce an intermediate map $\tilde{Q}_2$:

Let $k(x, y) = (x, \; -r x + y)$, with $a_{20} \ne 0$.
Then $k Q_2$ is of form
$\tilde{Q}_2=(a_{20} x^2 + a_{11} x y + a_{02} y^2,\; b_{11} x y)$, where $a_{20}$ for $\tilde{Q}_2$ equals 
$a_{20}$ for $Q_2$, and is therefore still nonzero.

Continuing:
Let $k(x,y)=(x/a_{20}, \; y)$. 
Then $k \tilde{Q}_2$ is of form $Q_3$.

\item $Q_3 \rightarrow Q_4=(x^2 + a_{02} y^2,\;  b_{11} x y)$: 
Do if $b_{11}$ for $Q_3$ is nonzero:

Let $k(x, y) = (x - (a_{11}/b_{11})y, \;y)$. 
Then $k Q_3$ is of form $Q_4$.

($b_{11}$ for $Q_4$ is the same as $b_{11}$ for $Q_3$, and hence, still nonzero.)

\item $Q_4 \rightarrow E_2=(x^2 - y^2, xy)$. Do if $a_{02}$ for $Q_4$ is negative.
Introduce an intermediate map $\tilde{Q}_4$.
Let $h(x,y)=(x,y/\sqrt{-a_{02}})$.
Then $Q_4 h$ is of form $\tilde{Q}_4$.

($b_{11}$ for $\tilde{Q}_4$ equals $b_{11}/\sqrt{-a_{02}}$ for $Q_4$,  and hence, still nonzero.)

Continuing, let $k(x,y)=(x,y/b_{11})$.
Then $k \tilde{Q}_4$ is of form
$E_2$.

\item $Q_4 \rightarrow P_3=(x^2,xy)$.
Do if $a_{02}$ for $Q_4$ is zero.

Then  $Q_4=(x^2, b_{11}xy)$.
Let $k(x,y)=(x,y/b_{11})$.
Then $k Q_4=P_3$.

\item $Q_4 \rightarrow H_3=(x^2+y^2,xy)$.
Do if $a_{02}$ for $Q_4$ is positive.

Introduce an intermediate map $\tilde{Q_4}$.
Let $h(x,y)=(x,y/\sqrt{a_{02}})$.
Then $Q_4 h$ is of form $\tilde{Q_4}=(x^2+y^2,b_{11}xy)$.

($b_{11}$ for $\tilde{Q_4}$ equals $b_{11}/\sqrt{a_{02}}$ for $Q_4$,  and hence, still nonzero.)

Continuing, let $k(x,y)=(x,y/b_{11})$.
Then $k \tilde{Q_4}$ is of form $H_3$.

\item $Q_3 \rightarrow Q_5=(x^2+a_{02}y^2, 0)$.
Do if $b_{11}$ for $Q_3$ is zero:

Let $k(x,y)=(x - a_{11}y/2, y)$.
Then $k Q_3$ is of form $Q_5$.

\item $Q_5 \rightarrow D^E_3=(x^2-y^2,0)$.
Do if $a_{02}$ for $Q_5$ is negative.

Let $h(x,y)=(x,y/\sqrt{-a_{02}})$.
Then $Q_5 h=D^E_3$.

\item $Q_5 \rightarrow D^P_5=(x^2,0)$.
Do if $a_{02}$ for $Q_5$ is zero.
$D^P_5$ is already equal to $Q_5$.
 
\item $Q_5 \rightarrow D^H_2=(x^2+y^2,0)$.
Do if $a_{02}$ for $Q_5$ is positive.
Let  $h(x,y)=(x,y/\sqrt{a_{02}})$.
Then $Q_5 h=D^H_2$.

\end{itemize}

Note that all six termination points in the algorithm are without parameters and appear in the final list of eighteen equivalence class representatives.

\subsection{Proof of Lemma \ref{lemma-elliptic}, 
Elliptic case}
This Lemma treats all cases with homogeneous quadratic part: $(x^2-y^2, xy)$.

\begin{table}[ht]
\label{table-ell-org}
\begin{tabular}{ccccccccc} 
$Q_0$&$\rightarrow$&$Q_1$&$\rightarrow$ &$Q_2$ &$\rightarrow$ &$\framebox[.8cm][l]{$E_1$}$\\
&&$\downarrow$&$\searrow$\\
&&$\framebox[.9cm][l]{$E_2$}$&&$Q_3$&$\rightarrow$&$\framebox[.8cm][l]{$E_1$}$ \\

\end{tabular}
\caption{Organization of the coordinate changes to prove Lemma \ref{lemma-elliptic}.  The boxed forms are ``normal forms'' that appear in the list of eighteen equivalence class representatives in Table \ref{table-ame}. All have the same homogeneous part.}
\end{table}

\begin{itemize}

\item Begin with $Q_0(x,y)=(x^2-y^2+a_{10}x+a_{01}y, xy+b_{10}x+b_{01}y)$.
\item $Q_0 \rightarrow Q_1=(x^2-y^2+a_{10}x, xy+b_{10}x)$: 
Let $h(x, y) = (x - b_{01},\; y + a_{01}/2)$.
Then $Q_0 h$ is of form $Q_1$.

\item $Q_1 \rightarrow Q_2=(x^2-y^2+x, xy+b_{10}x)$.
Do if $a_{10} \ne 0$.
Let $h(x, y) = (\frac{1}{a_{10}} x,  \frac{1}{a_{10}} y)$, and
$k(x, y) = (\frac{1}{a_{10}^2} x , \frac{1}{a_{10}^2} y )$.
Then $k Q_1 h^{-1}(x, y) = k Q_1(a_{10} x,  a_{10} y) 
= k(a_{10}^2(x^2 - y^2 + x), a_{10}^2 xy + a_{10} b_{10} x)
= (x^2 - y^2 + x, xy + (b_{10}/a_{10}) x)$ 
which is a map of the form $Q_2$.

\item $Q_2 \rightarrow E_1=(x^2-y^2+x, xy)$.
This is most complicated transformation to determine,
so we separate its proof as its own Lemma \ref{lemma-ell-longcase} below.

\item $Q_1 \rightarrow Q_3=(x^2-y^2, xy+x)$.
Do if $a_{10} = 0$ and $b_{10} \ne 0$.

Let $h(x, y) = (\frac{1}{b_{10}} x,  \frac{1}{b_{10}} y)$, and
$k(x, y) = (\frac{1}{b_{10}^2} x , \frac{1}{b_{10}^2} y )$ which are two invertible affine maps.
Then, a straight forward computation yields $Q_3 h = k Q_1$.

\item $Q_3 \rightarrow E_1=(x^2-y^2+x, xy)$.
Let $h(x, y) = (\frac{-1}{2}y - \frac{1}{2},  \frac{1}{2}x)$, and
$k(x, y) = (\frac{-1}{4}x - \frac{1}{4}, \frac{-1}{4}y )$ which are
invertible affine maps.
A straight forward computation yields $ E_1 h = k Q_3$.

\item $Q_1 \rightarrow E_2=(x^2-y^2, xy)$.
Do if $a_{10} = 0$ and $b_{10} = 0$.
Then $Q_1$ already equals $E_2$.

\end{itemize}

This, along with the proof of Lemma \ref{lemma-ell-longcase} below, completes the proof of Lemma \ref{lemma-elliptic}.
$\Box$

\begin{lemma}
\label{lemma-ell-longcase}
Let $Q_2(x,y)=(x^2-y^2+x, xy+b_{10}x)$, and
$E_1(x,y)=(x^2-y^2+x, xy)$.
Then $Q_2$ and $E_1$ are affinely map equivalent.
\end{lemma}

\begin{proof}
The goal is to eliminate the $b_{10}x$ term from $Q_2$.
Let $h(x, y) = (p_0 x + q_0 y + u_0, r_0 x + s_0 y + v_0)$, and
$k(x, y) = (p x + q y + u, r x + s y + v)$ be two invertible affine maps.
We shall show the existence, for any fixed $b_{10} \in \mR$, of $h$ and $k$ such that $E_1 h = k Q_2$.
This one vector polynomial equation is satisfied if and only if all 
corresponding terms in each component match.
This leads to twelve equations in the twelve unknown coefficients of $h$ and $k$.

Since  $h$ and $k$ are invertible affine maps,
\begin{equation} \label{e1} 
p_0s_0 - q_0 r_0 \neq 0
\end{equation} 
and  
\begin{equation} \label{e2}
ps - q r \neq 0.
\end{equation}
Then by comparing coefficients of $E_1 h = k Q_2$,
we have the system of equations:

\begin{eqnarray}
p_0^2 - r_0^2 & = & p  \label{ellipse_1} \\
2 p_0 q_0 - 2 r_0 s_0 & = & q \label{ellipse_2} \\
q_0^2 - s_0^2 & = & -p \label{ellipse_3} \\
2 p_0 u_0 - 2 r_0 v_0 + p_0 & = & b_{10} q + p \label{ellipse_4} \\
2 q_0 u_0 - 2 s_0 v_0 + q_0 & = & 0 \label{ellipse_5} \\
u_0^2 - v_0^2 + u_0 & = & u \label{ellipse_6} \\
p_0 r_0 & = & r \label{ellipse_7} \\
p_0 s_0 + q_0 r_0 & = & s \label{ellipse_8} \\
q_0 s_0 & = & -r \label{ellipse_9} \\
p_0 v_0 + r_0 u_0 & = & b_{10} s + r \label{ellipse_10} \\
q_0 v_0 + s_0 u_0 & = & 0 \label{ellipse_11} \\
u_0 v_0 & = & v \label{ellipse_12}
\end{eqnarray}

Note that variables $u$ and $v$ appear in one equation each:
equations (\ref{ellipse_6}) and (\ref{ellipse_12}).
Thus the remaining ten equations can be used to solve for the other ten variables.
The values of $u$ and $v$ can then be determined by these two equations.

Broadly speaking, we use eight of the equations to solve for the eight variables $q_0, u_0, r_0, s_0, v_0, p,q,s$ in terms of the remaining two variables $p_0$ and $r$.
Substituting these eight variables leaves us with two
coupled equations in $p_0$ and $r$.
Finally, we can solve for $r$ in terms of $p_0$, and we are left with a single (cubic) equation in $p_0$.
We do not find an explicit solution for $p_0$, but we use 
the intermediate value theorem to show a solution exists.
Then back substitution can be used to solve for the other eleven variables.
The existence of a solution establishes the equivalence of 
$Q_2$ and $E_1$.  Details of this computation follow.

(i) Claim $p_0 \neq 0$:
If $p_0 = 0$ then $r = 0$ by (\ref{ellipse_7}) and $q_0, r_0 \neq 0$ by (\ref{e1}).
Hence $s_0 = 0$ by (\ref{ellipse_9}). Hence $v_0 = 0$ by (\ref{ellipse_11}).
Since $p_0 = 0$ and $s_0 = 0$, we have $q = 0$ by (\ref{ellipse_2}).
Hence $r_0^2 = q_0^2 = -p$ by (\ref{ellipse_1}) and (\ref{ellipse_3}).
This forces $r_0=q_0=p=0$.
Therefore we end up with $r_0 = q_0 = 0$, a contradiction.

(ii) Claim $s_0 \neq 0$: If $s_0 = 0$ then $r = 0$ by (\ref{ellipse_9}) and
$q_0, r_0 \neq 0$ by (\ref{e1}). Hence $p_0 = 0$ by (\ref{ellipse_7}), a
contradiction to previous paragraph.

(iii) Since $p_0 \neq 0$ and $s_0 \neq 0$,
from (\ref{ellipse_7}) and (\ref{ellipse_9}), we have

\begin{equation}\label{e3}
r_0 = \frac{r}{p_0}
\end{equation}
 and 
\begin{equation}\label{e4}
q_0 = \frac{-r}{s_0}.
\end{equation}

Hence $p_0^2 - (\frac{r}{p_0})^2 = p$ by (\ref{ellipse_1}) and
$(\frac{-r}{s_0})^2 - s_0^2 = -p$ by (\ref{ellipse_3}).
Simplified, we see that both $p_0^2$ and $s_0^2$ are positive
root of the quadratic equation: $x^2 - px -r^2 = 0$.
Hence 
\begin{equation}\label{e5}
p_0^2 = s_0^2 = \frac{p+\sqrt{p^2 + 4r^2}}{2}
\end{equation}

We choose $p_0 = s_0 > 0$ (solutions like $p_0=-s_0>0$ might also be possible, but we need only to show the existence of one solution).

From (\ref{ellipse_5}) and  (\ref{ellipse_11}) we have:
\begin{eqnarray*}
2 q_0 u_0 - 2 s_0 v_0  + q_0 & = & 0  \\
s_0 u_0 + q_0 v_0  & = & 0 
\end{eqnarray*}

Solve this linear system with unknowns $u_0$ and $v_0$,
we get  
\begin{equation}\label{e6}
v_0 = \frac{s_0 q_0}{2(q_0^2 + s_0^2)} 
\end{equation}
and
\begin{equation}\label{e7}
u_0 =  \frac{- q_0^2}{2(q_0^2 + s_0^2)}  
\end{equation}

Then $r_0 = -q_0 = \frac{r}{p_0}$, 
$u_0 = \frac{- r^2}{2(r^2+p_0^4)}$,
$v_0 = \frac{- p_0^2 r}{2(r^2+p_0^4)}$.
Now, from (\ref{ellipse_2}), we have
$ q = 2 p_0 q_0 - 2 r_0 s_0 = 2 p_0 \frac{-r}{s_0}  - 2 \frac{r}{p_0} s_0
= -4r$. From (\ref{ellipse_8}), $s = p_0 s_0 +q_0 r_0 = \frac{-(r^2-p_0^4)}{p_0^2}$.
Finally $p = p_0^2 - r_0^2 = \frac{-(r^2-p_0^4)}{p_0^2}$ (\ref{ellipse_1}). Note that
$p=s$.

{\bf Two equations and two unknowns}
The eight equations we have used so far are used to substitute eight variables in terms of $p_0$ and $r$.

The unused equations are (\ref{ellipse_4}) and  (\ref{ellipse_10}). Reorganizing slightly and relabeling these equations, we have
 
\begin{eqnarray}
2 p_0 u_0 - 2 r_0 v_0 + p_0 - p & = & b_{10}q  \label{ellipse2_1} \\
p_0 v_0 + r_0 u_0 -r = b_{10} s \label{ellipse2_2}
\end{eqnarray}
This leaves us with the following coupled system of two equations in the two variables $p_0$ and $r$:

\begin{eqnarray}
r^2 - p_0^4 + p_0^3 + 4r b_{10} p_0^2 & = & 0 \label{ellipse2_3} \\
2b_{10} (r^2 - p_0^4) -(2 p_0 + 1) p_0 r & = & 0 \label{ellipse2_4}
\end{eqnarray}

{\bf One equation and one unknown!}
Compute
(\ref{ellipse2_4}) $- 2b_{10}$(\ref{ellipse2_3}).
This is a linear function in $r$, which we can solve for $r$ to get:
\begin{equation}\label{ellipse_r}
r = - \frac{2 b_{10} p_0^2}{(8 b_{10}^2 + 2) p_0 + 1}. 
\end{equation}

Plug back into (\ref{ellipse2_3}), and simplify, we get  
$(64 b_{10}^4 + 32 b_{10}^2 + 4) p_0^6 + (-12 b_{10}^2-3) p_0^4 - p_0^3 = 0$.

Since $p_0$ is required to be positive, we can factor out $p_0^3$. We are left trying to find a positive root to

\begin{equation}\label{ellipse_p_0}
f(p_0) := (64 b_{10}^4 + 32 b_{10}^2 + 4) p_0^3 + (-12 b_{10}^2-3) p_0 - 1 = 0.
\end{equation}

Note that $f(0) = -1 < 0$.
Since the cubic coefficient is equal to 
$(8b_{10}^2+2)^2$, and this is positive for any real $b_{10}$, then $f(p_0) >0$ if $p_0$ is large enough.

Therefore, by the intermediate value theorem, for any given $b_{10}\in \mR$, $f$ must have at least one positive real solution $p_0$ for (\ref{ellipse_p_0}).

Backward substitution gives us the existence of a solution for all twelve variables, and therefore for the transformations $h$ and $k$.
Note that 
$p_0  s_0 - q_0 r_0 = p_0^2 + r_0^2 > 0$ and 
$p s - q r = p^2 + 4 r^2 >0$ . Hence the transformations $h$ and $k$ are nonsingular  affine maps.

That is, $Q_2$ is affinely equivalent to $E_1$ for any $b_{10} \in \mR$.

\end{proof}

\subsection{Proof of Lemma \ref{lemma-hyperbolic} Hyperbolic case}

This Lemma treats all cases with homogeneous quadratic part: $(x^2+y^2, xy)$.

\begin{table}[ht]
\label{table-hyp-org}
\begin{tabular}{ccccccccc} 
&&&&$Q_3$&$\rightarrow$&$\framebox[.8cm][l]{$H_1$}$\\
&&&$\nearrow$\\

$Q_0$&$\rightarrow$&$Q_1$&$\rightarrow$ &$Q_2$ &$\rightarrow$ &$\framebox[.8cm][l]{$H_1$}$\\
&&$\downarrow$&&&$\searrow$\\
&&$\framebox[.9cm][l]{$H_3$}$&&&&$\framebox[.8cm][l]{$H_2$}$ \\

\end{tabular}
\caption{Organization of the coordinate changes to prove Lemma \ref{lemma-hyperbolic}.  The boxed forms are ``normal forms'' that appear in the list of eighteen equivalence class representatives in Table \ref{table-ame}. All have the same homogeneous part.}
\end{table}

\begin{itemize}

\item Begin with $Q_0(x,y)=(x^2+y^2+a_{10}x+a_{01}y, xy+b_{10}x+b_{01}y)$.

\item $Q_0 \rightarrow Q_1=(x^2+y^2+a_{10}x, xy+b_{10}x)$. 
Let $h(x, y) = (x - b_{01},\; y - a_{01}/2)$.
Then $Q_0 h$ is of form $Q_1$.

\item $Q_1 \rightarrow Q_2=(x^2+y^2+x, xy+b_{10}x)$.
Do if $a_{10} \neq 0$.
Let $h(x, y) = (\frac{1}{a_{10}} x,  \frac{1}{a_{10}} y)$, and
$k(x, y) = (\frac{1}{a_{10}^2} x , \frac{1}{a_{10}^2} y )$.
Then $h^{-1} = (a_{10} x,  a_{10} y)$, hence
$k Q_1 h^{-1}(x, y) = k Q_1 (a_{10} x,  a_{10} y) 
= k(a_{10}^2(x^2 + y^2 + x), a_{10}^2 xy + a_{10} b_{10} x)
= (x^2 + y^2 + x, xy + (b_{10}/a_{10}) x)$ 
which is a map of the form $Q_2$.

\item $Q_2 \rightarrow H_1=(x^2+y^2+x, xy)$.
Do if $b_{10} \neq 0$ and $b_{10} \neq \pm \frac{1}{2}$.
This is most complicated transformation to determine,
so we separate its proof as its own Lemma \ref{lemma-hyp-longcase} below.

\item $Q_2 \rightarrow H_2=(x^2+y^2+x, xy+ \frac{1}{2}x)$.
Do if $b_{10} = \pm \frac{1}{2}$.
If $b_{10}=+\frac{1}{2}$, then $Q_2$ already is equal to $H_2$.
If $b_{10}=-\frac{1}{2}$,
Let $h(x, y) = (-y - \frac{1}{2},  x + \frac{1}{2})$, and
$k(x, y) = (x, -y)$ which are
invertible affine maps.
Then $H_2 h = k Q_2$.

\item $Q_1 \rightarrow Q_3=(x^2+y^2, xy+x)$.
Do if $a_{10} = 0$ and $b_{10} \ne 0$.
Let $h(x, y) = (\frac{1}{b_{10}} x,  \frac{1}{b_{10}} y)$, and
$k(x, y) = (\frac{1}{b_{10}^2} x , \frac{1}{b_{10}^2} y )$.
Then $Q_3 h = k Q_1 $.

\item $Q_3 \rightarrow H_1=(x^2+y^2+x, xy)$.
Let $h(x, y) = (\frac{-1}{2}y - \frac{1}{2},  \frac{-1}{2}x)$, and
$k(x, y) = (\frac{1}{4}x - \frac{1}{4}, \frac{1}{4}y )$.
Then $H_1 h = k Q_3 $.

\item $Q_1 \rightarrow H_3=(x^2+y^2, xy)$.
Do if $a_{10} = 0$ and $b_{10} = 0$.
Then $Q_1$ already equals $H_3$.

\end{itemize}

This, along with the proof of Lemma \ref{lemma-hyp-longcase} below, completes the proof of Lemma \ref{lemma-hyperbolic}.
$\Box$

\begin{lemma}
\label{lemma-hyp-longcase}

Let $Q_2=(x^2+y^2+x, xy+b_{10}x)$ and
$H_1=(x^2+y^2+x, xy)$.
Assume $b_{10}x \neq 0,\pm \frac{1}{2}$.
Then $Q_2$ and $E_1$ are affinely map equivalent.
\end{lemma}

\begin{proof}
Analogous to the proof of Lemma \ref{lemma-ell-longcase},
we try to solve for affine invertible maps $h(x, y) = (p_0 x + q_0 y + u_0, r_0 x + s_0 y + v_0)$, and
$k(x, y) = (p x + q y + u, r x + s y + v)$ such that $H_1 h = k Q_2$. The proof of this Lemma almost exactly parallels the proof of Lemma \ref{lemma-ell-longcase}, except for the condition that $b_{10} \ne \pm \frac{1}{2}$ or $0$.
Consequently, we skip the details in the substitutions which take us from the original twelve equations (matching corresponding terms in $H_1 h = k Q_2$) and twelve unknowns (the coefficients defining $h$ and $k$) through the reduction to a single cubic equation in $p_0$, analogous to eq.(\ref{ellipse_p_0}), which we can use to solve for a positive solution $p_0$ in terms of the fixed parameter $b_{10}$.

That is, we need a positive solution to 

\begin{equation}\label{hyper_p_0}
f(p_0) := (64 b_{10}^4 - 32 b_{10}^2 + 4) p_0^3 + (12 b_{10}^2 - 3) p_0 - 1 = 0.
\end{equation}

Note that
$f(0) = -1 < 0$.
Further, the cubic coefficient, which can be factored as 
$(8b_{10}^2-2)^2$ is positive unless $b_{10}=\pm \frac{1}{2}$.
Therefore, $f(p_0)$ is positive for sufficiently large $p_0$.
The intermediate value theorem now guarantees a positive root of  $f(p_0)$.

Once we have this positive solution for $p_0$, back substitution yields the full solution:

$r = \frac{2 b_{10} p_0^2}{(8 b_{10}^2 - 2) p_0-1}$,

$r_0 =\frac{r}{p_0} = \frac{ (2 b_{10} p_0)}{(8 b_{10}^2 - 2) p_0-1}$,

$q_0 = r_0 =\frac{ (2 b_{10} p_0)}{(8 b_{10}^2 - 2) p_0-1}$,

$p = p_0^2 + r_0^2 $,

$s = p = p_0^2 + r_0^2 $,

$q = 4 r =  \frac{8 b_{10} p_0^2}{(8 b_{10}^2 - 2) p_0-1}$,

$s_0 = p_0$,

$u_0 = \frac{- q_0^2}{2(q_0^2 - s_0^2)}$,

$v_0 = \frac{s_0 q_0}{2(q_0^2 - s_0^2)}$,

$u = v_0^2+u_0^2+u_0 $,

$v = u_0 v_0$.

For any given $b_{10} \neq 0,\pm\frac{1}{2} \in \mR$,
it can be shown that the denominators in the above expressions are nonzero, and that this solution is yields nonsingular affine maps $h$ and $k$
such that $H_1 h = k Q_2$.
That is $Q_2$ is affinely equivalent to $H_1$ for any $b_{10} \in \mR$, $b_{10} \ne \pm \frac{1}{2},0$.

\end{proof}

\subsection{Proof of Lemma \ref{lemma-parabolic}: Parabolic case}

This Lemma treats all cases with homogeneous quadratic part: $(x^2, xy)$.

\begin{table}[ht]
\label{table-par-org}
\begin{tabular}{ccccccccc} 
$Q_0$&$\rightarrow$&$Q_1$&$\rightarrow$ &$Q_2$ &$\rightarrow$ &$\framebox[.8cm][l]{$P_1$}$\\
&&$\downarrow$&$\searrow$\\
&&$\framebox[.9cm][l]{$P_3$}$&&$\framebox[.8cm][l]{$P_2$}$ \\

\end{tabular}
\caption{Organization of the coordinate changes to prove Lemma \ref{lemma-parabolic}.  The boxed forms are ``normal forms'' that appear in the list of eighteen equivalence class representatives in Table \ref{table-ame}. All have the same homogeneous part.}
\end{table}

\begin{itemize}

\item Begin with $Q_0(x,y)=(x^2+a_{10}x+a_{01}y, xy+b_{10}x+b_{01}y)$.

\item $Q_0 \rightarrow Q_1=(x^2+a_{01}y, xy+b_{01}y)$: 
Let $h(x, y) = (x - a_{10}/2,\; y - b_{10})$.
Then $Q_0 h$ is of form $Q_1$.

\item $Q_1 \rightarrow Q_2=(x^2+y, xy+b_{01}y)$.
Do if $a_{01} \neq 0$.
Let $h(x, y) = (\frac{1}{a_{01}} x,  \frac{1}{a_{01}} y)$, and
$k(x, y) = (\frac{1}{a_{01}^2} x , \frac{1}{a_{01}^2} y )$.
Then $h^{-1} = (a_{01} x,  a_{01} y)$, so
$k Q_1 h^{-1}(x, y) = k Q_1(a_{01} x,  a_{01} y) 
= k(a_{01}^2(x^2 + y), a_{01}^2 xy + a_{01} b_{01} y)
= (x^2 + y, xy + (b_{01}/a_{01}) x)$
which is of form $Q_2$.

\item $Q_2 \rightarrow P_1=(x^2+y, xy)$.

Let $h(x, y) = (x + \frac{b_{01}}{3},  \frac{-2 b_{01}}{3} x + y + \frac{2 b_{01}^2}{9})$, and
$k(x, y) = (x + \frac{b_{01}^2}{3}, \frac{-2 b_{01}}{3} x + y + \frac{2 b_{01}^3}{27})$.
Then $P_1 h = k Q_2 $.

\item $Q_1 \rightarrow P_2=(x^2, xy+y)$.
Do if $a_{01}=0$ and $b_{01} \neq 0$.
Let $h(x, y) = (\frac{1}{b_{01}} x, y)$, and
$k(x, y) = (\frac{1}{b_{01}^2} x, \frac{1}{b_{01}} y)$ which are
invertible affine maps.
Then $P_2 h = k Q_1$.

\item $Q_1 \rightarrow P_3=(x^2, xy)$.
Do if $a_{01} = 0$ and $b_{01} = 0$.
Then $Q_1$ already equals $P_3$.

\end{itemize}

This completes the proof of Lemma \ref{lemma-parabolic}.
$\Box$
\subsection{Proof of Lemma \ref{lemma-deg-elliptic}  $(x^2 - y^2,\;  0)$} 

This Lemma treats all cases with homogeneous quadratic part: $(x^2-y^2, 0)$.

\begin{table}[ht]
\label{table-deg-ell-org}
\begin{tabular}{ccccccccc} 
$Q_0$&$\rightarrow$&$Q_1$&$\rightarrow$ &$Q_2$ &$\rightarrow$ &$\framebox[.8cm][l]{$D^E_1$}$\\
&&$\downarrow$&$\searrow$&&$\searrow$\\
&&$\framebox[.9cm][l]{$D^E_3$}$&&$\framebox[.8cm][l]{$D^E_1$}$&&$\framebox[.8cm][l]{$D^E_2$}$ \\

\end{tabular}
\caption{Organization of the coordinate changes to prove Lemma \ref{lemma-elliptic}.  The boxed forms are ``normal forms'' that appear in the list of eighteen equivalence class representatives in Table \ref{table-ame}. All have the same homogeneous part.}
\end{table}

\begin{itemize}

\item Begin with $Q_0(x,y)=(x^2-y^2+a_{10}x+a_{01}y, b_{10}x+b_{01}y)$.

\item $Q_0 \rightarrow Q_1=(x^2-y^2, b_{10}x+b_{01}y)$.
Let $h(x, y) = (x - a_{10}/2,\; y + a_{01}/2)$.
Then $Q_0 h$ is of form $Q_1$.

\item $Q_1 \rightarrow Q_2=(x^2-y^2, x+b_{01}y)$.
Do if $b_{10} \neq 0$.
Let $k(x, y) = (x , y/b_{10} )$.
Then $k Q_1$ is of form $Q_2$.

\item $Q_2 \rightarrow D^E_1=(x^2-y^2,y)$.
Do if $b_{01} \neq \pm 1$.

Let $h(x, y) = (b_{01} x + y, x + b_{01} y)$, and
$k(x, y) = ( (b_{01}^2 - 1) x, y)$ which are
invertible affine maps if $b_{01} \neq \pm 1$.
Then $D^E_1 h = k Q_2$.

\item $Q_2 \rightarrow D^E_2=(x^2-y^2,x+y)$.
Do if $b_{01} = \pm 1$.
If $b_{01}=1$, then $Q_2$ is already equal to $D^E_2$.
If $b_{01}=-1$, let $h(x,y)=(x,-y)$.
Then $Q_2 h = D^E_2$.

\item $Q_1 \rightarrow D^E_1=(x^2-y^2, y)$.
Do if $b_{10}=0$ and $b_{01} \neq 0$.
Let $k(x, y) = (x , y/b_{01} )$.
Then $k Q_1 = D^E_1$.

\item $Q_1 \rightarrow D^E_3=(x^2-y^2,0)$.
Do if $b_{10} = 0$ and $b_{01} = 0$.
Then $Q_1$ already equals $D^E_3$.

\end{itemize}

This completes the proof of Lemma \ref{lemma-deg-elliptic}.
$\Box$


\subsection{ Proof of Lemma \ref{lemma-deg-hyperbolic} }

This Lemma treats all cases with homogeneous quadratic part: $(x^2+y^2, 0)$.
We start from $Q_0(x,y)=(x^2+y^2+a_{10}x +a_{01}y, b_{10}x +b_{01}y)$.

\begin{table}[ht]
\label{table-deg-hyp-org}
\begin{tabular}{ccccccccc} 
$Q_0$&$\rightarrow$&$Q_1$&$\rightarrow$ &$Q_2$ &$\rightarrow$ & $Q_3$&$\rightarrow$&$Q_4$\\
& &$\downarrow$&$\searrow$&&&&&$\downarrow$\\
& &$\framebox[.9cm][l]{$D^H_2$}$&
&$\framebox[.8cm][l]{$D^H_1$}$
&&&&$\framebox[.8cm][l]{$D^H_1$}$ \\

\end{tabular}
\caption{Organization of the coordinate changes to prove Lemma \ref{lemma-deg-hyperbolic}.  The boxed forms are ``normal forms'' that appear in the list of eighteen equivalence class representatives in Table \ref{table-ame}. All have the same homogeneous part.}
\end{table}

\begin{itemize}

\item Begin with $Q_0(x,y)=(x^2+y^2+a_{10}x+a_{01}y, b_{10}x+b_{01}y)$.

\item $Q_0 \rightarrow Q_1=(x^2+y^2, b_{10}x+b_{01}y)$.
Let $h(x, y) = (x - a_{10}/2,\; y - a_{01}/2)$.
Then $Q_0 h$ is of form $Q_1$.

\item $Q_1 \rightarrow Q_2=(x^2+y^2, x+b_{01}y)$.
Do if $b_{10} \neq 0$.
Let $k(x, y) = (x , y/b_{10} )$.
Then $k Q_1$ is of form $Q_2$.

\item $Q_2 \rightarrow Q_3=(x^2+y^2,x+y)$.
If $b_{01} = 1$, $Q_2 = Q_3$. They are affinely map equivalent.

If $b_{01} \neq 1$, let $h(x, y) = (x - \frac{b_{01}+1}{b_{01} - 1} y, - \frac{b_{01}+1}{b_{01} - 1}x - y)$, and
$k(x, y) = ( \frac{2(b_{01}^2+1)}{(b_{01} - 1)^2} x, -\frac{2}{b_{01} - 1} y)$.
Then $Q_3 h = k Q_2$.

\item $Q_3 \rightarrow Q_4=(x^2+y^2,x)$.
Note that $Q_4 = Q_2|_{b_{01}=0}$.
So the same $h$ and $k$ as in the previous step (also with $b_{01}=0$) gives $Q_3 h = k Q_4$.

\item $Q_4 \rightarrow D^H_1=(x^2+y^2, y)$.
Let $h(x,y)=(y,x)$.
Then $Q_4 h = D^H_1$.

\item $Q_1 \rightarrow D^H_1=(x^2+y^2, y)$.
Do if $b_{10}=0$ and $b_{01} \neq 0$.
Let $k(x, y) = (x , y/b_{01} )$.
Then $k Q_1 = D^H_1$.

\item $Q_1 \rightarrow D^H_2=(x^2+y^2,0)$.
Do if $b_{10} = 0$ and $b_{01} = 0$.
Then $Q_1$ already equals $D^H_2$.

\end{itemize}

This completes the proof of Lemma \ref{lemma-deg-hyperbolic}.
$\Box$

\subsection{Proof of Lemma \ref{lemma-deg-parabolic} }  

This Lemma treats all cases with homogeneous quadratic part: $(x^2, 0)$.
\begin{center}
\begin{table}[ht]
\label{table-deg-par-org}
\begin{tabular}{ccccccccc} 
&&&&$\framebox[.9cm][l]{$D^P_1$}$
&
&$\framebox[.9cm][l]{$D^P_4$}$\\
&&&&$\uparrow$&$\nearrow$\\
$Q_0$&$\rightarrow$&$Q_1$&$\rightarrow$ &$Q_2$ &$\rightarrow$
&$Q_3$&$\rightarrow$&$\framebox[.8cm][l]{$D^P_2$}$\\
&$\swarrow$&$\downarrow$&$\searrow$\\
$\framebox[.9cm][l]{$D^P_5$}$
&&$\framebox[.9cm][l]{$D^P_3$}$
&&$Q_4$&$\rightarrow$&$\framebox[.8cm][l]{$D^P_2$}$ \\

\end{tabular}
\caption{Organization of the coordinate changes to prove Lemma \ref{lemma-elliptic}.  The boxed forms are ``normal forms'' that appear in the list of eighteen equivalence class representatives in Table \ref{table-ame}. All have the same homogeneous part.}
\end{table}
\end{center}

\begin{itemize}

\item Begin with $Q_0(x,y)=(x^2+a_{10}x+a_{01}y, b_{10}x+b_{01}y)$.

\item $Q_0 \rightarrow Q_1=(x^2+a_{01}y, b_{10}x+b_{01}y)$.
Let $h(x, y) = (x - a_{10}/2,y)$.
Then $Q_0 h$ is of form $Q_1$.

\item $Q_1 \rightarrow Q_2=(x^2+a_{01}y, x+b_{01}y)$.
Do if $b_{10} \ne 0$.
Let $k(x, y) = (x, y/b_{10})$.
Then $k Q_1$ is of form $Q_2$.

\item $Q_2 \rightarrow Q_3=(x^2+a_{01}y,x+y)$.
Do if $b_{01} \ne 0$.
Let $h(x, y) = (x, y/b_{01})$.
Then $Q_2 h$ is of form $Q_3$.

\item $Q_3 \rightarrow D^P_2=(x^2,y)$.
Let $h(x, y) = (2x - a_{01}, x+y)$, and $k(x, y) = (4x - 4 a_{01} y + a_{01}^2, y)$.
Then $D^P_2 h = k Q_3$.

\item $Q_1 \rightarrow Q_4=(x^2+a_{01}y, y)$.
Do if $b_{10} =0$ and $b_{01}\ne 0$.
Let $k(x, y) = (x, y/b_{01})$.
Then $k Q_1 =Q_4$. 

\item $Q_4 \rightarrow D^P_2=(x^2,y)$.
Let $k(x, y) = (x -  a_{01} y, y)$.
Then $D^P_2 = k Q_4$.

\item $Q_2 \rightarrow D^P_1=(x^2+y, x)$.
Do if $b_{01}=0$ but $a_{01} \ne 0$.
Let $h(x, y) = (x, y/a_{01})$.
Then $Q_2 h =D^P_1$. 

\item $Q_1 \rightarrow D^P_3=(x^2+y, 0)$.
Do if $b_{10}=0, b_{01} =0$ and $a_{01}\ne 0$.
Let $h(x, y) = (x, y/a_{01})$.
Then $Q_1 h=D^P_3$.

\item $Q_2 \rightarrow D^P_4=(x^2, x)$.
Do if $a_{01} =0$ and $b_{01} = 0$.
Then $Q_2$ is already equal to $D^P_4$.

\item $Q_1 \rightarrow D^P_5=(x^2,0)$.
Do if $a_{01}, b_{10}$ and $b_{01}$ are all zero.
Then $Q_1$ is already equal to $D^P_5$.

\end{itemize}

This completes the proof of Lemma \ref{lemma-deg-parabolic}.
$\Box$


\section{Acknowledgments}
 Thanks to Bernd Krauskopf and Hinke Osinga for discussions related to this project and to a related project in preparation \cite {PKON} dealing with the parameter space for quadratic maps of the plane, and transitions  between affine map equivalence classes.
We acknowledge M. Golubitsky for pointing out the 
 language of algebraic orbits stated in Sec. \ref{ss-algebraic}.


\clearpage

\clearpage


\begin{thebibliography}{30}



\bibitem [Abraham {\it et al.}(1997)]{AGMbook}
Abraham, R. H., Gardini, L., \& Mira, C. [1997]
{\it Chaos in Discrete Dynamical Systems: A Visual Introduction in 2
  Dimensions} (Springer-Verlag, New York).







\bibitem[Bofill {\it et al.}(2004)]{BGVR}
Bofill, F., Garrido, J. L., Villamajo, F. \& Romero, N. [2004]
``On the Quadratic Endomorphisms of the Plane,''
{\it Advanced Nonlinear Studies}, {\bf 4}, pp.~37--55.

\bibitem[Delgado {\it et al.}(2013)]{DGRRV}
Delgado J., Garrido J. L., Romero, N., Rovella, A. \& Vilamajo, F, ``On the geometry of quadratic maps
of the plane,'' {it Publ. Mat. Uruguay}, Vol 14. (2013) 120–135.




\bibitem[Frouzakis {\it et al.}(1997)]{FGKMM} Frouzakis, C. E., Gardini, L., Kevrekidis, I. G., Millerioux, G., \& Mira, C. [1997] ``On some properties of of invariant sets of two-dimensional noninvertible maps,''  {\it Int. J. Bifurc. Chaos}, {\bf 7}(6), pp.~1167--1194.


\bibitem[Garrido {\it et al.}(2005)]{GRRV} Garrido, J. L., Romero, N., Rovella, A. \& Vilamajo, F. [2005] ``Critical Points of Quadratic Maps of the Plane,''  {\it PreMAT}, Prepublicationes de Matematica de la Universidad de la Republica de Uraguay, 2005/83.

\bibitem[Golubitsky \& Guillemin(1973)]{GG} Golubitsky, M. \& Guillemin, V. [1973]
{\it Stable mappings and thier singularities}, Graduate Texts in Mathematics, {\bf 14}, (Springer-Verlag, New York).



\bibitem[Gumowski \& Mira(1980b)]{GM}
Gumowski, I. \& Mira, C. [1980b]
{\it Recurrences and Discrete Dynamic Systems} 
(Springer-Verlag, New York).

\bibitem[H\'enon(1976)]{Henon} H\'enon, M. [1976]
``A two-dimensional mapping with a strange attractor,''
{\it Commun. Math. Phys.}, {\bf 50}(1), pp.~69--77

\bibitem[(1997)]{Nienthesis} Nien, C.-H. [1997] ``The investigation of saddle-node bifurcation with a zero eigenvalue -- includes example of non-analyticity,'' Ph.~D. thesis, University of Minnesota.
%
\bibitem[Nien(1998)]{Nien98}Nien C.-H. [1998] ``The dynamics of Planar Quadratic
Maps with Nonempty Bounded Critical Set,''
{\it Int. J. Bifurc. Chaos}, {\bf 8}(1), pp.~95--105.

\bibitem[Nien et al. (2016)]{NPM}Nien, C.-H., Peckham, B. B. \& McGehee, R. P. 
[2016]
``Classification of critical sets and their images for quadratic maps of the plane, {\it J Difference Equations and Applications}, {\bf 22}(5), pp637-655, http://dx.doi.org/10.1080/10236198.2015.1127360.

\bibitem[Peckham {\it et al.}(in prep.)]{PKON} Peckham, B. B., Krauskopf, B., Osinga, H. \& Nien C-H [in prep.] ``The parameter space for quadratic maps of the plane'', in preparation.



``Perturbations of the quadratic family of order two,'' {\it Nonlinearity} {\bf 14}, pp.~1633-1652.



\end{thebibliography}
\end{document}